\documentclass[11pt,lenq]{amsart}
\usepackage{amsmath}
\usepackage{amscd}
\usepackage{amssymb}
\usepackage{amsbsy}
\usepackage{amsfonts}
\usepackage{color}
\usepackage{latexsym}
\usepackage{graphics}
\usepackage{graphicx}
\usepackage{amsmath,amscd,latexsym}
\usepackage{array}
\usepackage{enumerate}
\theoremstyle{plain}
\newtheorem{theorem}{Theorem}[section]

\newtheorem{lemma}[theorem]{Lemma}

\newtheorem{main theorem}[theorem]{Main Theorem}

\newtheorem{question}[theorem]{Question}

\theoremstyle{definition}   

\newtheorem{definition}[theorem]{Definition}

\newlength\savewidth

\newcommand{\BS}{\mathrm{BS}}
\newcommand{\Conj}{\mathrm{Conj}}

\newcommand{\svert}{\,|\,}


\begin{document}
	\title[Connections between conjugation quandles and their underlying groups via residual finiteness and the Hopf property]
	{Connections between conjugation quandles and their underlying groups via residual finiteness and the Hopf property}
	\author{Mohamed Elhamdadi}
	\address{Mathematics and Statistics College of Arts and Sciences, University of South Florida, Tampa, FL 33620, USA}
	\email{emohamed@usf.edu}	
	\author{Jan Kim}
	\address{Institute of Mathematical Sciences, Ewha Womans University,
	52, Ewhayeodae-gil, Seodaemun-gu, Seoul, 03760 Republic of Korea}
	\email{jankim@ewha.ac.kr}		
	
	
	
	\begin{abstract}
        We prove that if a conjugation quandle is Hopfian, then its underlying group is also Hopfian.
        We also show that the converse does not hold by providing an example. This highlights a distinction between conjugation quandles and their underlying groups. While a recent result shows that every hyperbolic group is Hopfian, conjugation quandles of hyperbolic groups can still be non-Hopfian.
        Furthermore, we examine conjugation quandles of Baumslag–Solitar groups. We show that these quandles are infinitely generated. Hence, to apply the result that every finitely generated residually finite quandle is Hopfian, it is necessary to work with finitely generated quandles. For this purpose, we employ Dehn quandles as subquandles, which allow us to fully characterize the residual finiteness of conjugation quandles of the Baumslag–Solitar groups.
	\end{abstract}
	
	\maketitle
	
	\section{Introduction}
	\label{sec:introduction}		
	
	Quandles are algebraic structures based on Reidemeister moves in knot theory. They were introduced independently by Joyce~\cite{Joyce} and Matveev~\cite{Matveev}.
	They have been used in the construction of invariants for knots (see, for example,~\cite{Elhamdadi_Nelson}) and knotted surfaces (see~\cite{Carter}), as well as in the formulation of set-theoretic solutions to the Yang-Baxter equation (see~\cite{Bardakov_Elhamdadi_Singh, Carter_Elhamdadi_Saito}).
	The binary operations of quandles generalize the notion of conjugation relations. 
    Specifically, the Wirtinger presentation of the fundamental group of a link complement involves conjugation relations of the form $z=y^{-1}xy$. Also, replacing conjugation by a binary operation gives a presentation of the link quandle.
    Every group gives a quandle structure defined by the binary operation: $x * y = y^{-1}xy$, called a {\em conjugation quandle}.
	
	This article examines the residual finiteness of conjugation quandles. 
    An algebraic structure $\mathcal{A}$ is called {\em residually finite} if for all $x, y \in \mathcal{A}$ with $x \neq y$, there is a finite algebraic structure $\mathcal{P}$ and a homomorphism $\psi : \mathcal{A} \rightarrow \mathcal{P}$ such that $\psi(x) \neq \psi(y)$.
	The residual finiteness of groups provides an important connection between topology and group theory. A group $G$ is residually finite if and only if the profinite topology on $G$ is Hausdorff.
	All finitely generated 3-manifold groups, including link groups,
	have been shown to be residually finite (see~\cite{Hempel}).
	Similar to the residual finiteness of groups, this property can also be applied to quandles. It was first introduced in~\cite{Bardakov_Singh_Singh}, where the authors proved that conjugation quandles are residually finite if their underlying groups are residually finite.
	In particular, they subsequently showed that link quandles are residually finite (see~\cite{Bardakov_Singh_Singh_2}). 
    The residual finiteness allows passage from infinite link quandles to finite quandles, thus making computations of invariants arising from coloring of links simpler.
	
	Exploring the residual finiteness of conjugation quandles builds a bridge between geometric group theory and quandle theory. In geometric group theory, there is a challenging open problem posed by Gromov~\cite{Gromov}, concerning whether every hyperbolic group is residually finite. Since conjugation quandles are residually finite if their underlying groups are residually finite (see~\cite{Bardakov_Singh_Singh}), the existence of a non-residually finite conjugation quandle with a hyperbolic underlying group would imply the existence of a non-residually finite hyperbolic group.
	Moreover, a common strategy is to use a less restrictive property to approach Gromov's question, specifically by utilizing the Hopf property. An algebraic structure is called {\em Hopfian} if every surjective endomorphism of the structure is an automorphism.
	Mal'cev~\cite{Malcev} proved that every finitely generated residually finite group is Hopfian.
	Since hyperbolic groups are finitely generated, researchers have examined Gromov's question through the lens of the Hopf property. This perspective raised the question of whether there exist non-Hopfian hyperbolic groups. Recently, Reinfeldt and Weidmann~\cite{Reinfeldt_Weidmann} proved that all hyperbolic groups are Hopfian. Subsequently, Fujiwara and Sela~\cite{Fujiwara_Sela} provided an alternative proof, confirming this result. Although non-Hopfian hyperbolic groups do not exist, the Hopf property of quandles, which remains an area of limited exploration, still offers a potential approach to addressing Gromov's question.
	In a way analogous to Mal'cev's result, it is shown in~\cite{Bardakov_Singh_Singh} that every finitely generated residually finite quandle is Hopfian. Moreover, it is established in~\cite{Bardakov_Singh_Singh} that conjugation quandles are residually finite if their underlying groups are residually finite, and the residual finiteness is closed under taking subquandles.
    Thus, one direct approach to Gromov's question is that: if all hyperbolic groups are residually finite, then every finitely generated subquandle of a conjugation quandle with a hyperbolic underlying group is Hopfian.
	
	In this article, we prove that if a conjugation quandle is Hopfian, then its underlying group is also Hopfian in Lemma~\ref{thm:Hopfian_conjugation_quandle}.
    We also show that the converse does not hold in Theorem~\ref{thm:non_Hopfian_Conj(G)}. The latter result highlights a distinction between conjugation quandles and their underlying groups. Although a recent result~\cite{Reinfeldt_Weidmann, Fujiwara_Sela} shows that every hyperbolic group is Hopfian, our Theorem~\ref{thm:non_Hopfian_Conj(G)} shows that conjugation quandles of hyperbolic groups can be non-Hopfian.
    Famous examples dealing with the residual finiteness and the Hopf property are the Baumslag-Solitar groups. Hence, to analyze the connection between quandles and groups through the residual finiteness and the Hopf property, we examine conjugation quandles of the Baumslag–Solitar groups. Lemma~\ref{thm:Hopfian_conjugation_quandle} gives particular cases of the conjugation quandles of the Baumslag-Solitar groups which are non-Hopfian. However, it is difficult to use the fact that every finitely generated residually finite quandle is Hopfian since we show that conjugation quandles of the Baumslag–Solitar groups are infinitely generated in Theorem~\ref{thm:infinitely_generated}. Hence, to connect the residual finiteness to the Hopf property of quandles, it is necessary to work with finitely generated quandles. For this purpose, we employ Dehn quandles as subquandles of conjugation quandles, which allows us to fully characterize the residual finiteness of conjugation quandles of the Baumslag–Solitar groups in Theorem~\ref{thm:Main_Theorem}.
	
	\section{Basics of Quandles} \label{sec:preliminaries}
	
	In this section, we recall the necessary definitions and known facts of quandles.
	
	\begin{definition}[Quandles]
		\label{def:quandles}
		A non-empty set $Q$ equipped with a binary operation $* : Q \times Q \rightarrow Q$ is called a {\em quandle} if it satisfies the following three conditions:
		\begin{enumerate}[\rm (i)]
			\item $x * x = x$ for all $x \in Q$.
			\item For $x, y \in Q$, there exists a unique $z \in Q$ such that
			$x = z * y$.
			\item $(x * y) * z = (x * z) * (y * z)$ for all $x, y, z \in Q$. 			
		\end{enumerate}			
	\end{definition}
	
	Any set $X$ with a binary operation satisfying $x * y=x$, for all $x, y \in X$, is a quandle called a {\em trivial quandle}. Notice that the three conditions of a quandle $(Q, *)$ directly come from the Reidemeister moves of classical knot theory. 
	In axiom (ii) of Definition~\ref{def:quandles}, the element $z \in Q$ such that $x = z * y$ can be expressed by $x *^{-1} y$. 
	Clearly, the {\em inverse operation} $*^{-1}$ also satisfies (i)-(iii), and thus the set $Q$ equipped with the binary operation $*^{-1} : Q \times Q \rightarrow Q$ is a quandle as well. 
	
	In the remainder of the article, we will use 
	\[
	x_1 *^{\epsilon_1} x_2 *^{\epsilon_2} \cdots *^{\epsilon_{k-1}} x_k
	\]
	to denote 
	\[
	( \cdots ((x_1*^{\epsilon_1}x_2)*^{\epsilon_2}x_3)* \cdots)*^{\epsilon_{k-1}} x_k
	\]
	in a quandle, where $\epsilon_i=\pm 1$ for each $i=1, \dots, k-1$.
	
	Quandles can be derived from groups. Especially, every group can be associated with a conjugation quandle as follows.
	
	\begin{definition}[Conjugation quandles]
	Given a group $G$, the set $G$ equipped with the binary operation $x * y = y^{-1}xy$ which gives a quandle structure on $G$ is called the {\em conjugation quandle} of $G$, and is denoted by $\Conj(G)$.		
	\end{definition}
	
	Note that even though the underlying groups of conjugation quandles are finitely generated, this does not necessarily mean that the conjugation quandles themselves are finitely generated. 
	We will discuss this further in Section~\ref{sec:Conjugation_quandles}.
	\begin{definition}[Dehn quandles]
		Let $G$ be a group, $X$ a nonempty subset of $G$ and $X^G$ the set of all conjugates of elements of $X$ in $G$.
		We call {\em Dehn quandle} of $G$ with respect to $X$, denoted by $\mathcal{D}(X^G)$, the set $X^G$ equipped with the binary operation $x * y = y^{-1}xy$ which gives a quandle structure on $\mathcal{D}(X^G)$.	
	\end{definition}

Clearly, for each $X \subseteq G$, Dehn quandle $\mathcal{D}(X^G)$ is a subquandle of $\Conj(G)$. Also, every subquandle of $\Conj(G)$ is a Dehn quandle (see~\cite[Corollary~3.9]{Dhanwani_1}). Moreover, if a group $G$ is generated by $X$, then $\mathcal{D}(X^G)$ is generated as a quandle by $X$ (see~\cite[Proposition~3.2]{Dhanwani_1}).
Note that a Dehn quandle is finitely generated when the subset $X$ of the group $G$ is a finite set.
	
\section{Connection between conjugation quandles and their underlying groups through the Hopf property}
\label{sec:Conjugation_quandles}

Conjugation quandles and their underlying groups are distinct algebraic structures that nonetheless share several similarities.
There are relations between the residual finiteness and the Hopf property for both groups and quandles. Mal'cev~\cite{Malcev} proved that every finitely generated residually finite group is Hopfian. In~\cite{Bardakov_Singh_Singh}, it was shown that every finitely generated residually finite quandle is Hopfian, and if $G$ is residually finite, then $\Conj(G)$ is residually finite. 
In this section, we show that if $\Conj(G)$ is Hopfian then $G$ is Hopfian. We also prove that the converse does not hold.
We start by stating the following lemma which will also be used in Section 4.

\begin{lemma}
\label{thm:Hopfian_conjugation_quandle}
Let $G$ be a group. If the quandle $\Conj(G)$ is Hopfian, then the group $G$ is Hopfian.		
\end{lemma}

\begin{proof}
	Suppose in contrast that $G$ is non-Hopfian.
	Then there exists an epimorphism $\psi : G \rightarrow G$ such that $\psi$ is not injective, which implies that $\psi$ induces the quandle epimorphism $\tilde{\psi} : \Conj(G) \rightarrow \Conj(G)$ which is not injective.
	Hence $\Conj(G)$ is non-Hopfian, which is a contradiction.
	Thus, $G$ is Hopfian.		
\end{proof}
It is natural to ask whether the converse of Lemma~\ref{thm:Hopfian_conjugation_quandle} holds. In other words, is there a Hopfian conjugation quandle $\Conj(G)$ such that its underlying group $G$ is non-Hopfian?
The following theorem gives an answer to this question:

\begin{theorem}
\label{thm:non_Hopfian_Conj(G)}
Let $G=\langle a \rangle$ be an infinite cyclic group generated by $\{a\}$.
Then $\Conj(G)$ is non-Hopfian, whereas $G$ is Hopfian.	
\end{theorem}

\begin{proof}    
To show that $\Conj(G)$ is non-Hopfian, let $\phi : \Conj(G) \rightarrow \Conj(G)$ be a quandle map defined by $1 \mapsto 1$, $a^n \mapsto a^{n-1}$ and $a^{-n} \mapsto a^{-n+1}$ for every $n \in \mathbb{Z}_+$.
Since the group $G$ is abelian, $\Conj(G)$ is a trivial quandle. Hence, any map from $\Conj(G)$ to itself is a homomorphism. Hence, $\phi$ is also a homomorphism. Clearly, $\phi$ is surjective.
But since $\phi(1)=1=\phi(a)$ with $a \neq 1$ in $\Conj(G)$, $\phi$ is not injective.
Thus, $\Conj(G)$ is non-Hopfian. 	
\end{proof}

Theorem~\ref{thm:non_Hopfian_Conj(G)} highlights a distinction between conjugation quandles and their underlying groups. Recall that every hyperbolic group is Hopfian~\cite{Fujiwara_Sela, Reinfeldt_Weidmann}.
However, since every cyclic group is hyperbolic, Theorem~\ref{thm:non_Hopfian_Conj(G)} gives an example of a non-Hopfian conjugation quandle whose underlying group is hyperbolic.
Moreover, since every trivial quandle is residually finite, Theorem~\ref{thm:non_Hopfian_Conj(G)} implies that there exists an infinitely generated quandle which is non-Hopfian but residually finite.

\section{The Conjugation Quandles of the Baumslag-Solitar Groups}

In this section, we examine conjugation quandles of the Baumslag–Solitar groups in order to highlight their similarities and to discuss some differences.
One of the most famous examples of a finitely presented non-Hopfian group is the Baumslag-Solitar group $\BS(2,3)=\langle a, b \svert b^{-1}a^2b=a^3 \rangle$. By Lemma~\ref{thm:Hopfian_conjugation_quandle}, the conjugation quandle $\Conj(\BS(2,3))$ is non-Hopfian.
In general, the Baumslag–Solitar groups $\BS(m,n)=\langle a, b \svert b^{-1}a^mb=a^n \rangle$, where $m, n \in \mathbb{Z} \setminus \{0\}$, provide a particularly useful case study, as their Hopfian property and residual finiteness are completely classified in terms of their defining relations.
The classification of the residual finiteness and the Hopf property for 
$\BS(m,n)$ is as follows, based on~\cite{Baumslag_Solitar} and corrected by Meskin~\cite{Meskin}:

\begin{itemize}
	\item $\BS(m,n)$ is residually finite if and only if $|m|=1$ or $|n|=1$ or $|m|=|n|$.
	\item $\BS(m,n)$ is Hopfian if and only if it is residually finite or $\pi(m)=\pi(n)$, where $\pi(m)$ and $\pi(n)$ denote the set of all primes dividing $m$ and $n$, respectively.	
\end{itemize}

Using Lemma~\ref{thm:Hopfian_conjugation_quandle}, one obtains that the conjugation quandle $\Conj(\BS(m,n))$ is non-Hopfian if $|m| \neq 1 \neq |n|$ and $\pi(m) \neq \pi(n)$, where $\pi(m)$ and $\pi(n)$ denote the set of all primes dividing $m$ and $n$, respectively.

The following is a complete classification of the residual finiteness for conjugation quandles of the Baumslag-Solitar groups. A detailed proof of the following theorem will be provided in the latter part of this section.
	
	\begin{theorem}\label{thm:Main_Theorem}
		Let $\BS(m,n)$ be the Baumslag-Solitar group given by the presentation
		\[
		\BS(m,n)=\langle a, b \svert b^{-1}a^mb=a^n \rangle,
		\]
		where $m, n \in \mathbb{Z} \setminus \{0\}$.
		Then the conjugation quandle $\Conj(\BS(m,n))$ is residually finite if and only if $|m|=1$, $|n|=1$ or $|m|=|n|$.
	\end{theorem}
	
	Every finitely generated residually finite quandle is Hopfian (see~\cite[Theorem~5.7]{Bardakov_Singh_Singh}). However, the fact that conjugation quandles are non-Hopfian does not necessarily imply that they are always non-residually finite.
The fact that the underlying groups of conjugation quandles are finitely generated does not guarantee that the conjugation quandles themselves are finitely generated (see Theorem~\ref{thm:non_Hopfian_Conj(G)}, for example). 
We prove that the conjugation quandles of the Baumslag-Solitar groups are also infinitely generated.
	
	\begin{theorem}
		\label{thm:infinitely_generated}
		Let $\BS(m,n)$ be the Baumslag-Solitar group given by the presentation
		\[
		\BS(m,n)=\langle a, b \svert b^{-1}a^mb=a^n \rangle,
		\]
		where $m, n \in \mathbb{Z} \setminus \{0\}$.
		Then the conjugation quandle $\Conj(\BS(m,n))$ is infinitely generated.	
	\end{theorem}
	
	\begin{proof}
It is well-known that two conjugate elements in a group have the same image through the abelianization. 
Consider the abelianization $\BS(m,n)_{ab}$ of $\BS(m,n)$.
Then $gb^kg^{-1} \neq b^{\ell}$ in $\BS(m,n)_{ab}$ for each $k \neq \ell$, and hence $b^k$ and $b^\ell$ are not conjugate if $k \neq \ell$.
Therefore, the number of elements in the generating set of $\Conj(\BS(m,n))$ is not bounded,
this completes the proof of Theorem~\ref{thm:infinitely_generated}.	\end{proof}
	
	Now we give the proof of Theorem~\ref{thm:Main_Theorem}.
	
	\begin{proof}[Proof of Theorem~\ref{thm:Main_Theorem}]
		Suppose that $|m|=1$, $|n|=1$ or $|m|=|n|$.
		Then, since $\BS(m,n)$ with $|m|=1$, $|n|=1$ or $|m|=|n|$ is residually finite, by \cite[Proposition~4.1]{Bardakov_Singh_Singh}, $\Conj(\BS(m,n))$ is residually finite.	
		
		Now, suppose that $1 \neq |m| \neq |n| \neq 1$. We categorize this assumption into three cases.
		In each case, we will derive that $\Conj(\BS(m,n))$ is non-residually finite.
		In this proof, let $\pi(m)$ and $\pi(n)$ denote the set of all primes dividing $m$ and $n$, respectively.
		
		\medskip
		\noindent {\bf Case 1.} {\em $\pi(m) \neq \pi(n)$.}
		
		\medskip
		Let $\mathcal{D}$ be the Dehn quandle of $\BS(m,n)$ with respect to $\{a, b, a^m\}$.
		We will show that $\mathcal{D}$ is non-residually finite, which implies that $\Conj(\BS(m,n))$ is also non-residually finite  due to the fact that the residual finiteness is subquandle-closed.
		To be specific, we will prove that $\mathcal{D}$ is finitely generated and non-Hopfian, thereby proving that $\mathcal{D}$ is non-residually finite. We will first establish the following claim.
		
		\medskip
		\noindent {\bf Claim.} {\em The quandle $\mathcal{D}$ is finitely generated.}
		
		\begin{proof}[Proof of Claim]
			Let $w \in \mathcal{D}$ be an arbitrary element. Then $w=g^{-1}xg$ for some $x \in \{a, b, a^m\}$ and $g \in \BS(m,n)$. If $g=1$ in $\BS(m,n)$, then clearly $w=x \in \{a, b, a^m \}$. Assume that $g \neq 1$ in $\BS(m,n)$. Since $g \in \BS(m,n)$, we may write $g \equiv x_1^{\epsilon_1} \cdots x_k^{\epsilon_k}$ such that $x_i \in \{a, b\}$ and $\epsilon_i=\pm 1$ for each $i=1, \dots, k$.
			Hence $w$ can be written as 
			$x_k^{-\epsilon_k} \cdots x_1^{-\epsilon_1} x\; x_1^{\epsilon_1} \cdots x_k^{\epsilon_k}$, where $x \in \{a, b, a^m\}$, $x_i \in \{a, b\}$, and $\epsilon_i=\pm 1$ for each $i=1, \dots, k$.
			In other words, $w$ can be expressed as either
			\[
			a *^{\epsilon_1} x_1 *^{\epsilon_2} \cdots *^{\epsilon_k} x_k, \quad b *^{\epsilon_1} x_1 *^{\epsilon_2} \cdots *^{\epsilon_k} x_k, \quad \text{or} \quad a^m *^{\epsilon_1} x_1 *^{\epsilon_2} \cdots *^{\epsilon_k} x_k,
			\]
			where $x_i \in \{a, b\}$ and $\epsilon_i=\pm 1$ for each $i=1, \dots, k$.
			Thus, $\mathcal{D}$ is generated by $\{a, b, a^m\}$.		
		\end{proof}
		
		Now we verify that $\mathcal{D}$ is non-Hopfian.
		Let $\phi$ be the unique quandle homomorphism from a free quandle on $\{a, b, a^m\}$ to $\mathcal{D}$ induced by the mapping
		\[
		a \mapsto a^m, \quad b \mapsto b \quad \text{and} \quad a^m \mapsto a.
		\]
		Then $\phi$ induces an endomorphism $\tilde{\phi}$ of $\mathcal{D}$ and clearly, $\tilde{\phi}$ is surjective. However, $\tilde{\phi}$ is not injective since
		$\tilde{\phi}(a * b * a *^{-1} b)=\tilde{\phi}(a)$ but $a * b * a *^{-1} b \neq a$ in
		$\mathcal{D}$ since $ba^{-1}b^{-1}abab^{-1} \neq a$ in $\BS(m,n)$ by Britton's Lemma. 
		Therefore, $\mathcal{D}$ is non-residually finite, and thus, for Case~1,
		$\Conj(\BS(m,n))$ is non-residually finite. 
		
		\medskip
		\noindent {\bf Case 2.} {\em Either $m \mid n$ or $n \mid m$, and $\pi(m)=\pi(n)$.}
		
		\medskip
		Without loss of generality,  we may assume that $n=\ell m$ for some $\ell \in \mathbb{Z} \setminus \{0, \pm 1\}$.
		Let $\mathcal{D}$ be the Dehn quandle of $\BS(m,n)$ with respect to $\{a, b, a^m\}$.
		We will demonstrate that $\mathcal{D}$ is non-residually finite by applying a similar argument to that used in Case~1.
		In analogy with the conclusion drawn in Case~1, it follows that $\mathcal{D}$ is finitely generated.
		
		Now we prove that $\mathcal{D}$ is non-Hopfian.
		Let $\phi$ be the unique homomorphism from a free quandle on $\{a, b, a^m\}$ to $\mathcal{D}$ induced by the mapping
		\[
		a \mapsto a, \quad b \mapsto b \quad \text{and} \quad a^m \mapsto a^m * b.
		\]
		So $\phi$ induces an endomorphism $\tilde{\phi}$ of $\mathcal{D}$.
		Clearly, $\tilde{\phi}$ is surjective. However $\tilde{\phi}$ is not injective since
		\[
		\tilde{\phi}(a^m *^{-1} b * a)=\tilde{\phi}(a^m *^{-1} b)
		\]
		but $a^m *^{-1} b * a \neq a^m *^{-1} b$ in $\mathcal{D}$, since $a^{-1}ba^mb^{-1}a \neq ba^{-m}b^{-1}$ in $\BS(m,n)$ by Britton's Lemma. Thus, for Case 2, $\Conj(\BS(m,n))$ is non-residually finite.
		
		\medskip
		\noindent {\bf Case 3.} {\em Neither $m \mid n$ nor $n \mid m$, while $\pi(m)=\pi(n)$.}
		
		\medskip
		In this case, since $\pi(m)=\pi(n)$, there exists $k=\gcd(m,n)$ with $k \neq 1$.
		Then
		\[
		\pi\left(\frac{m}{k}\right) \neq \pi\left(\frac{n}{k}\right) \quad \text{and} \quad \left|\frac{m}{k}\right| \neq 1 \neq \left|\frac{n}{k}\right|.
		\]
		Hence, by the conclusion of Case~1, $\Conj(\BS(\frac{m}{k}, \frac{n}{k}))$ is non-residually finite. Since $\BS(\frac{m}{k}, \frac{n}{k})$ is a subgroup of $\BS(m,n)$,
		and since the residual finiteness is subgroup-closed, 
		it follows that $\Conj(\BS(m,n))$ is non-residually finite for Case~3. This end the proof of Theorem~\ref{thm:Main_Theorem}.			
	\end{proof}
	
	
	By Theorem~\ref{thm:infinitely_generated}, the conjugation quandle $\Conj(\BS(m,n))$ is always infinitely generated. Thus, it is not possible to observe the Hopf property using Theorem~\ref{thm:Main_Theorem} based on the fact that finitely generated residually finite quandles are Hopfian (see~\cite{Bardakov_Singh_Singh}).
	In this regard, we leave the following question.
	
	\begin{question}
	If $m,n \in \mathbb{Z} \setminus \{0\}$ with $|m|=1$, $|n|=1$, or $|m|=|n|$, is the conjugation quandle $\Conj(\BS(m,n))$ Hopfian?	
	\end{question}
	
	Finally, in order to achieve a complete classification of both the residual finiteness and the Hopf property of conjugation quandles of the Baumslag-Solitar groups, we pose the following question.
	
	\begin{question}
		If $m,n \in \mathbb{Z} \setminus \{0\}$ with $\pi(m)=\pi(n)$, where $\pi(m)$ and $\pi(n)$ denote the set of all primes dividing $m$ and $n$, respectively,
		then is the conjugation quandle $\Conj(\BS(m,n))$ Hopfian?	
	\end{question}
	
	\section*{Acknowledgement} 
	Mohamed Elhamdadi was partially supported by Simons Foundation collaboration grant 712462.
    Jan Kim was supported by Basic Science Research Program through the National Research Foundation of Korea(NRF) funded by the Ministry of Education (2021R1A6A1A10039823).

\end{document}